\documentclass[11pt]{article}

\usepackage[margin=1in]{geometry}
\usepackage{amsmath,amssymb,amsthm,mathtools}
\usepackage{booktabs}
\usepackage{enumitem}
\usepackage{longtable}
\usepackage{placeins}
\usepackage{tikz}
\usepackage[colorlinks=true,linkcolor=blue,citecolor=blue,urlcolor=blue]{hyperref}

\usetikzlibrary{calc,positioning}

\newtheorem{theorem}{Theorem}[section]
\newtheorem{proposition}{Proposition}[section]
\newtheorem{lemma}{Lemma}[section]
\newtheorem{corollary}{Corollary}[section]
\theoremstyle{definition}
\newtheorem{definition}{Definition}[section]
\newtheorem{example}{Example}[section]
\newtheorem{remark}{Remark}[section]

\newcommand{\kap}{\kappa}
\newcommand{\drawsmallcat}[3]{%
\begin{scope}[scale=.55]
  \foreach \a [count=\i from 0] in {#3} {
    \node[vertex] (v\i) at (.82*\i,0) {};
    \ifnum\i>0
      \pgfmathtruncatemacro{\prev}{\i-1}
      \draw (v\prev)--(v\i);
    \fi
  }
  \foreach \a [count=\i from 0] in {#3} {
    \ifnum\a>0
      \foreach \j in {1,...,\a} {
        \pgfmathsetmacro{\mid}{(\a+1)/2}
        \ifnum\i=0
          \pgfmathsetmacro{\ang}{180+(\j-\mid)*22}
        \else
          \pgfmathtruncatemacro{\last}{#2-1}
          \ifnum\i=\last
            \pgfmathsetmacro{\ang}{(\j-\mid)*22}
          \else
            \pgfmathsetmacro{\ang}{90+(\j-\mid)*22}
          \fi
        \fi
        \node[vertex] (l\i-\j) at ($(v\i)+(\ang:.62)$) {};
        \draw (v\i)--(l\i-\j);
      }
    \fi
  }
  \node[rectangle,draw=none,fill=none,font=\scriptsize] at ({.41*(#2-1)},-.95) {#1};
\end{scope}%
}

\title{Positive-Curvature Discrete Einstein Metrics on Trees}
\author{Haoxuan Cheng\thanks{School of Mathematical Sciences, Fudan University. Email: \texttt{hxcheng25@m.fudan.edu.cn}.}}
\date{}

\begin{document}
\maketitle

\begin{abstract}
For a weighted tree, the Lin--Lu--Yau Ricci curvature admits an explicit
formula in terms of the edge weights. Consequently, the constant-curvature
equation is equivalent to an eigenvalue problem for an edge-indexed Ricci
matrix $R_T$. Building on the spectral characterization of discrete Einstein
metrics on trees, we classify all finite trees whose discrete Einstein metric
has positive curvature, equivalently all trees satisfying
$\lambda_{\max}(R_T)<0$. For caterpillars with spine order $m\ge 12$, this
occurs precisely for the endpoint families $T_m(a,0,\ldots,0,b)$ with
$1\le a,b\le 3$ and $(a,b)\ne(3,3)$. The remaining cases $3\le m\le 11$ are
settled by an exact finite verification using rational characteristic
polynomials and Sturm root counts. We also determine the zero level set
$\lambda_{\max}(R_T)=0$: among caterpillars, it consists of the stable family
$(3,0,\ldots,0,3)$ together with nine exceptional short-spine caterpillars,
while $S_3^2$ is the unique non-caterpillar zero example.
\end{abstract}

\section{Introduction}

The problem of finding Einstein metrics is one of the central problems in
Riemannian geometry \cite{besse}. A Riemannian metric $g$ is Einstein if
\[
  \operatorname{Ric}(g)=\kap g
\]
for a constant $\kap$. In the smooth setting, the Einstein equation is closely
related to the fixed points of the normalized Ricci flow
\cite{hamilton,chow-lu-ni}.

Several discrete analogues of Ricci curvature have been introduced on graphs,
including Ollivier Ricci curvature and the Lin--Lu--Yau curvature
\cite{ollivier,lly}; see also \cite{bauer-jost-liu,jost-liu}. These curvatures
are connected with analytic and geometric properties such as spectrum,
diameter, and rigidity; see for instance
\cite{lin-yau-eigen,paeng,munch-general}. They have also led to applications
in graph coloring, community detection, and network geometry, as well as to
structural results for idleness functions, graph products, and special graph
families; see
\cite{lin-yau-coloring,bourne-idleness,munch-general,cushing-products,cushing-cubic,
samal-networks,sia-community,gosztolai-dynamical,hoorn-rgg,
hehl-lly-one,hehl-positive-regular}.

Let $G=(V,E,w)$ be a finite weighted graph with positive edge weight
$w:E\to\mathbb R_{>0}$. For an edge $xy\in E$, denote by $\kap_{xy}$ its
Lin--Lu--Yau Ricci curvature. We call $(G,w)$ a discrete Einstein graph if
$\kap_{xy}$ is independent of $xy$. For trees, the curvature admits a closed
formula. After multiplying by the edge weights, the discrete Einstein condition
therefore becomes a finite-dimensional linear eigenvalue problem.

In addition to static curvature, one can also study geometric evolution on
graphs. For weighted graphs, Ollivier Ricci flow has been introduced and
analyzed in \cite{bai-ricci-sum,bai-ollivier-flow}, while for trees a detailed
long-time analysis was recently obtained in \cite{baihua-ricci-flow}. In the
recent work of Bai and Hua \cite{baihua-dem}, the associated edge-indexed
matrix $R_T$ was introduced. We call $R_T$ the Ricci matrix of the tree.

Perron--Frobenius theory yields a unique positive eigenvector of $R_T$, up to
scaling. This eigenvector is precisely the discrete Einstein metric on the
tree. Thus every finite tree admits a unique discrete Einstein metric, while
the sign of its curvature is governed by the largest eigenvalue of $R_T$.
Equivalently, the positive-curvature problem becomes the problem of classifying
finite trees satisfying
\[
  \lambda_{\max}(R_T)<0.
\]

The starting point of the present paper is the implication proved in
\cite{baihua-dem}:
\[
  \lambda_{\max}(R_T)<0
  \quad\Longrightarrow\quad
  T\text{ is a caterpillar}.
\]
That paper also established the spectral characterization of discrete Einstein
metrics on trees, several attachment monotonicity results, and the stable zero
family obtained from the symmetric double star; see
\cite[Theorem 1.2, Propositions 3 and 7, Proposition 8]{baihua-dem}. The
converse of the implication is false. The natural next question, posed in
\cite{baihua-dem}, is therefore to determine exactly which caterpillars satisfy
$\lambda_{\max}(R_T)<0$, or equivalently, which trees carry a
positive-curvature discrete Einstein metric.

For a caterpillar, we write $T_m(a)$ for the tree with spine order $m$ and
leaf-count sequence
\[
  a=(a_1,\ldots,a_m),\qquad a_1,a_m\ge1,\quad a_i\ge0,
\]
where $a_i$ denotes the number of leaves attached to the $i$th spine vertex.

Our main result is the following theorem.

\begin{theorem}\label{thm:intro-main}
Let $T$ be a finite tree.
\begin{enumerate}[label=(\roman*)]
  \item The discrete Einstein metric on $T$ has positive curvature if and only
  if $\lambda_{\max}(R_T)<0$. If $\lambda_{\max}(R_T)<0$, then $T$ is a
  caterpillar.

  \item Let $T=T_m(a)$ be a caterpillar. Then $\lambda_{\max}(R_T)<0$ if and
  only if one of the following holds:
  \begin{enumerate}[label=(\alph*)]
    \item $m=1$;
    \item $m=2$ and $(a_1-1)(a_2-1)<4$;
    \item $3\le m\le11$ and $a$ lies in the downward closure of the maximal
    elements listed in Table~\ref{tab:maximal-negative};
    \item $m\ge12$ and
    \[
      a=(a_1,0,\ldots,0,a_m),\qquad
      1\le a_1,a_m\le3,\qquad
      (a_1,a_m)\ne(3,3).
    \]
  \end{enumerate}

  \item One has $\lambda_{\max}(R_T)=0$ if and only if one of the following
  holds:
  \begin{enumerate}[label=(\alph*)]
    \item $T=S_3^2$;
    \item $T=T_m(a)$ is a caterpillar with
    \[
      a=(3,0,\ldots,0,3);
    \]
    \item $T=T_m(a)$ is a caterpillar and, up to reversal, $a$ is one of the
    parameters listed in Table~\ref{tab:zero-exceptions}.
  \end{enumerate}
\end{enumerate}
\end{theorem}

Thus Theorem~1.1 completes the classification problem posed in
\cite{baihua-dem}.

In Section~2, we restate the negative and zero classifications separately as
Theorems~\ref{thm:main} and \ref{thm:zero}, after introducing the Ricci matrix
and the caterpillar quotient equations.

The contribution of the present work splits naturally into four parts. The cases $m=1$ and
$m=2$, namely stars and double stars, are treated separately. For caterpillars
with long spine, namely $m\ge12$, we combine a Rayleigh-quotient argument for
leaf attachment with an explicit Schur-complement analysis to reduce the
problem to the endpoint families. The remaining finite range $3\le m\le11$ is
then resolved by an exact finite verification, recorded in the tables below.
The same leaf-attachment argument also identifies the unique non-caterpillar
tree in the zero level set, namely $S_3^2$. Since the Einstein curvature is
$\kap=-\lambda_{\max}(R_T)$, this is exactly a classification of
positive-curvature discrete Einstein metrics on trees.

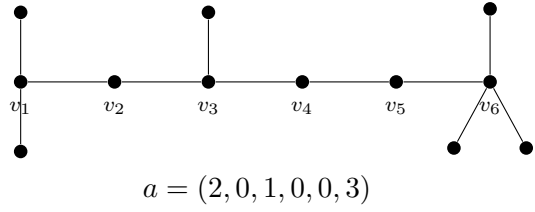
\begin{figure}[t]
\centering
\begin{tikzpicture}[scale=.92, every node/.style={circle,fill=black,inner sep=1.8pt}]
  \foreach \i in {1,...,6} {
    \node (v\i) at (1.35*\i,0) {};
    \node[draw=none,fill=none,rectangle,inner sep=1pt] at (1.35*\i,-.38) {\scriptsize $v_\i$};
  }
  \foreach \i in {1,...,5} {
    \pgfmathtruncatemacro{\j}{\i+1}
    \draw (v\i)--(v\j);
  }
  \node (l11) at (1.35,1.0) {};
  \node (l12) at (1.35,-1.0) {};
  \draw (v1)--(l11);
  \draw (v1)--(l12);
  \node (l31) at (4.05,1.0) {};
  \draw (v3)--(l31);
  \node (l61) at (8.10,1.05) {};
  \node (l62) at (7.58,-.95) {};
  \node (l63) at (8.62,-.95) {};
  \draw (v6)--(l61);
  \draw (v6)--(l62);
  \draw (v6)--(l63);
  \node[draw=none,fill=none,rectangle] at (4.75,-1.55)
    {$a=(2,0,1,0,0,3)$};
\end{tikzpicture}
\caption{The caterpillar $T_6(2,0,1,0,0,3)$ encoded by the leaf-count sequence
$a=(2,0,1,0,0,3)$.}
\end{figure}

The paper is organized as follows. Section~2 introduces the Ricci matrix, the
caterpillar quotient equations, and the main classification statements.
Section~3 gives the leaf-attachment argument, the initial cases $m=1,2$, and
the non-caterpillar reduction. Section~4 treats the long-spine caterpillar
case. Section~5 handles the remaining short-spine range by exact finite
verification. Section~6 discusses the infinite-tree setting. The appendices
contain the finite boundary tables, representative zero vectors, and the
boundary-generation check.

\section{Preliminaries}

Throughout the paper, all trees are finite, simple, connected, and undirected.
For a tree $T=(V,E)$ and a vertex $v\in V$, let $d_v$ be the degree of $v$.
A vertex of degree one is called a leaf or pendant vertex. An edge incident to a
leaf is called a leaf edge. An edge whose two endpoints both have degree at
least two is called an internal edge.

\begin{definition}[Star and caterpillar]
For $k\ge2$, a star $S_k$ is a tree with exactly one internal vertex and $k$
edges; we also include the degenerate case $S_1=K_2$. A caterpillar is a tree for
which removing all leaves leaves a path. This remaining path is called the
spine. In our parametrization, the spine vertices are denoted
$v_1,\ldots,v_m$, and $a_i$ is the number of leaves attached to $v_i$.
\end{definition}

We write $S_3^2$ for the tree obtained by subdividing each edge of the
three-leaf star $S_3$ once. Equivalently, $S_3^2$ has a central vertex of
degree three, and each of its three incident edges is extended by one
additional edge to a leaf.

We recall that for a caterpillar in canonical form one writes
\[
  a=(a_1,\ldots,a_m),\qquad a_1,a_m\ge1,\quad a_i\ge0.
\]
The corresponding tree is denoted by $T_m(a)$. Reversing $a$ gives an isomorphic
tree, so tables below list only one representative up to reversal. The case
$m=1$ is the star $S_{a_1}$ and is treated separately below.

\subsection{The Ricci curvature formula on trees}

Let $T=(V,E,w)$ be a weighted tree with $w_e>0$ for every edge $e$.
For a vertex $v$, set
\[
  S_v=\sum_{u\sim v} w_{uv}.
\]
For an edge $e=xy$, the Lin--Lu--Yau curvature on a tree is
\[
  \kap_{xy}
  =
  -\left(
    \frac{S_x-2w_{xy}}{d_xw_{xy}}
    +
    \frac{S_y-2w_{xy}}{d_yw_{xy}}
  \right).
\]
Therefore the condition $\kap_{xy}\equiv\kap$ is equivalent, after multiplying
the equation by $w_{xy}$, to
\[
  -\kap\, w_{xy}
  =
  \frac{S_x-2w_{xy}}{d_x}
  +
  \frac{S_y-2w_{xy}}{d_y}.
\]
This is precisely an eigenvalue equation for the Ricci matrix below, introduced
in \cite[Definition 1]{baihua-dem}.

\begin{definition}[Ricci matrix]
For a finite tree $T=(V,E)$, define the edge-indexed matrix
$R_T\in\mathbb R^{E\times E}$ by
\[
  (R_T)_{e,e}=-\left(\frac1{d_x}+\frac1{d_y}\right),
  \qquad e=\{x,y\},
\]
and, for distinct edges $e\ne f$,
\[
  (R_T)_{e,f}=
  \begin{cases}
    1/d_z,& e\cap f=\{z\},\\
    0,& e\cap f=\varnothing.
  \end{cases}
\]
\end{definition}

For an edge weight vector $w=(w_e)_{e\in E}$, the above curvature formula gives
\[
  R_Tw=\lambda w,\qquad \lambda=-\kap.
\]
The matrix $R_T$ is real symmetric and Metzler. Since the line graph $L(T)$ is
connected, $R_T+cI$ is irreducible and nonnegative for all sufficiently large
$c$. Perron--Frobenius theory implies that $\lambda_{\max}(R_T)$ is simple and
has a strictly positive eigenvector. Consequently every finite tree has a unique
discrete Einstein metric up to scaling, and its curvature is
\[
  \kap=-\lambda_{\max}(R_T).
\]
This recovers the spectral characterization proved in
\cite[Theorem 1.2]{baihua-dem}.

\begin{remark}[Schr\"odinger viewpoint]
Equivalently, $R_T$ may be viewed as a weighted Schr\"odinger operator on the line
graph $L(T)$. This viewpoint is useful near the zero level set
$\lambda_{\max}=0$, where degree-two chains satisfy a zero-energy second-order
recurrence.
\end{remark}

\subsection{The caterpillar quotient equations}

We record the finite system used throughout the classification. In this
subsection assume $m\ge2$, since the case $m=1$ is a star and is treated
separately. Let $T_m(a)$ be a caterpillar with spine vertices
$v_1,\ldots,v_m$.Write
$s_j=v_jv_{j+1}$ for the spine edge between $v_j$ and $v_{j+1}$, and put
\[
  q_i=
  \begin{cases}
    1,&i=1,m,\\
    2,&2\le i\le m-1,
  \end{cases}
  \qquad
  d_i=q_i+a_i .
\]
On the Perron eigenspace, all pendant edges attached to the same spine vertex
have the same value. We denote the value on $s_j$ by $x_j$, and the common value
of the pendant edges at $v_i$ by $y_i$ when $a_i>0$. Set
\[
  X_i=
  {\bf 1}_{i>1}x_{i-1}+{\bf 1}_{i<m}x_i+a_i y_i .
\]
The eigenvalue equation $R_Tw=\lambda w$ is equivalent to
\[
  \lambda x_j
  =
  \frac{X_j-2x_j}{d_j}
  +
  \frac{X_{j+1}-2x_j}{d_{j+1}},
  \qquad 1\le j\le m-1,
\]
and, for every $i$ with $a_i>0$,
\[
  \lambda y_i
  =
  -y_i+\frac{X_i-2y_i}{d_i}.
\]
Let $\mathcal U$ denote the subspace of edge functions that are constant on each
orbit of sibling pendant edges. The preceding equations are the matrix form of
$R_T|_{\mathcal U}$ in the natural, generally non-orthonormal common-value
coordinates. Since $R_T$ commutes with permutations of sibling pendant edges and
the Perron eigenvalue is simple, the Perron eigenvector lies in $\mathcal U$.
Hence $\lambda_{\max}(R_T)$ is the largest eigenvalue of this restricted
coordinate matrix.

This coordinate matrix has the correct spectrum on $\mathcal U$, but is not
symmetric in general. For Schur complements, inertia, and Sylvester's criterion,
one must pass to an orthonormal orbit basis.For every $i$ with $a_i>0$, let
\[
  u_i=\frac1{\sqrt{a_i}}\sum_{e\text{ pendant at }v_i} e
\]
be the normalized pendant-orbit vector, and write
\[
  z_i=\sqrt{a_i}\,y_i .
\]
Equivalently, if $M_m(a)$ is the common-value coordinate matrix and $D$ is the
diagonal matrix with entries $1$ on spine coordinates and $\sqrt{a_i}$ on the
pendant coordinate at $v_i$, then
\[
  Q_m(a)=D M_m(a)D^{-1}.
\]

In the orthonormal basis consisting of the spine edges and the vectors $u_i$,
the restricted matrix becomes a real symmetric matrix $Q_m(a)$, and the
eigenvalue equations take the form
\[
  \lambda x_j
  =
  -\left(\frac1{d_j}+\frac1{d_{j+1}}\right)x_j
  +\frac{{\bf 1}_{j>1}}{d_j}x_{j-1}
  +\frac{{\bf 1}_{j<m-1}}{d_{j+1}}x_{j+1}
  +\frac{\sqrt{a_j}}{d_j}z_j
  +\frac{\sqrt{a_{j+1}}}{d_{j+1}}z_{j+1},
\]
where the terms involving $z_j$ or $z_{j+1}$ are omitted when the
corresponding $a$-entry is zero, and
\[
  \lambda z_i
  =
  \frac{\sqrt{a_i}}{d_i}\bigl({\bf 1}_{i>1}x_{i-1}+{\bf 1}_{i<m}x_i\bigr)
  -\frac{q_i+2}{d_i}z_i .
\]
The matrices used in Sections 4 and 5 are always these symmetric orthonormal
orbit matrices. This is the setting in which Schur complements and
Sylvester's criterion apply directly.

\medskip
\noindent We can now state the main classification theorems.

For caterpillar parameters of the same spine order, define the coordinatewise
partial order by $a\le b$ if $a_i\le b_i$ for all $i$. Two parameters are said
to be coordinatewise comparable if either $a\le b$ or $b\le a$. A subset
$\mathcal A$ of caterpillar parameters is called downward closed if
$b\in\mathcal A$ and $a\le b$ imply $a\in\mathcal A$, and upward closed if
$a\in\mathcal A$ and $a\le b$ imply $b\in\mathcal A$.
For a subset $\mathcal A$ of caterpillar parameters, an element $a\in\mathcal A$
is called maximal in $\mathcal A$ if there is no $b\in\mathcal A$ with
$a<b$.

\begin{theorem}[Classification of positive-curvature trees]\label{thm:main}
Let $T$ be a finite tree. Then the discrete Einstein metric on $T$ has positive
curvature if and only if $\lambda_{\max}(R_T)<0$. Such a tree must be a
caterpillar. More precisely, the negative region of caterpillars is as follows.
\begin{enumerate}[label=(\roman*)]
  \item For $m=1$, i.e. for stars $S_k$, every $k\ge1$ satisfies
  $\lambda_{\max}(R_{S_k})<0$.
  \item For $m=2$, i.e. for double stars $T_2(a_1,a_2)$,
\[
  \lambda_{\max}(R_{T_2(a_1,a_2)})<0
  \quad\Longleftrightarrow\quad
  (a_1-1)(a_2-1)<4.
\]
  \item For $3\le m\le11$, the negative region is the downward closure of the
  maximal elements listed in Table \ref{tab:maximal-negative}.
  \item For $m\ge12$,
\[
  \lambda_{\max}(R_{T_m(a)})<0
\]
if and only if
\[
  a=(a_1,0,\ldots,0,a_m),\qquad
  1\le a_1,a_m\le3,\qquad (a_1,a_m)\ne(3,3).
\]
\end{enumerate}
\end{theorem}

\begin{theorem}[The zero level set]\label{thm:zero}
Among non-caterpillar trees with $\lambda_{\max}(R_T)=0$, the only example is
$S_3^2$. Every tree that properly contains $S_3^2$ as a connected subtree has
$\lambda_{\max}(R_T)>0$.
Among caterpillars, the zero level set consists of the stable family
\[
  (3,0,\ldots,0,3),
\]
obtained by subdividing the central edge of the double star $S_{3,3}$, together
with the finite list in Table \ref{tab:zero-exceptions}, up to reversal.
\end{theorem}

\begin{remark}
The proof separates the structural reductions from the finite verification.
The stable zero family $(3,0,\ldots,0,3)$ already appears in
\cite[Proposition 8]{baihua-dem}; the new point here is the complete exhaustion
of both the negative region and the zero level set.
\end{remark}

\section{Preliminary Reductions}

This section collects the preliminary reductions used throughout the
classification. The main ingredient is a Rayleigh-quotient argument for leaf
attachment, based on a useful vertex decomposition of the Ricci matrix.

For an edge function $f:E(T)\to\mathbb R$, put
\[
  S_v(f)=\sum_{e\ni v}f_e,\qquad
  A_v(f)=\sum_{e\ni v}f_e^2.
\]
A direct expansion of the definition of $R_T$ gives the vertex decomposition
\[
  \langle f,R_Tf\rangle
  =
  \sum_{v\in V(T)}
  \frac1{d_v}\left(S_v(f)^2-2A_v(f)\right).
\]
Indeed, at a fixed vertex $v$,
\[
  2\sum_{\{e,e'\}\subset E_v}f_ef_{e'}
  =
  \left(\sum_{e\in E_v}f_e\right)^2-\sum_{e\in E_v}f_e^2,
\]
and combining this with the diagonal contribution yields the displayed formula.

Once the implication
\[
  \lambda_{\max}(R_T)<0 \Longrightarrow T \text{ is a caterpillar}
\]
is known, the non-caterpillar sign consequences can be recovered after the
present classification is established. Moreover, in view of the leaf-attachment
monotonicity established in the previous paper, compare
\cite[Propositions 3 and 7]{baihua-dem}, the following conclusion may also be
derived from the earlier results. We nevertheless record an independent
Rayleigh-quotient proof, both for completeness and because it is the form used
later in the zero-layer argument.

\begin{proposition}[Monotonicity of the nonnegative region under leaf attachment]
\label{prop:upper-closed}
Let $T'$ be obtained from $T$ by attaching one pendant edge at a vertex $v$.
If
\[
  \lambda_{\max}(R_T)\ge0,
\]
then
\[
  \lambda_{\max}(R_{T'})\ge0.
\]
Moreover, if $\lambda_{\max}(R_T)>0$, then
$\lambda_{\max}(R_{T'})>0$.
\end{proposition}

\begin{proof}
Let $d=d_v$. Extend an old edge function $f$ to $T'$ by assigning value $y$ to
the new pendant edge. Only the local contribution at $v$ and the new leaf
endpoint changes. Put
\[
  S=\sum_{e\ni v}f_e,\qquad A=\sum_{e\ni v}f_e^2.
\]
Then
\[
\begin{aligned}
\Delta(y)
&=\langle \widetilde f,R_{T'}\widetilde f\rangle
  -\langle f,R_Tf\rangle\\
&=
  -\frac{S^2-2A}{d(d+1)}
  +\frac{2Sy}{d+1}
  -\frac{d+2}{d+1}y^2 .
\end{aligned}
\]
This quadratic polynomial is maximized at $y=S/(d+2)$, and
\[
  \max_y\Delta(y)
  =
  \frac2{d(d+1)}
  \left(A-\frac{S^2}{d+2}\right)\ge0
\]
by Cauchy's inequality. If $R_T$ has a nonnegative Rayleigh quotient, the
extended vector can be chosen so that
\[
  \langle \widetilde f,R_{T'}\widetilde f\rangle
  \ge
  \langle f,R_Tf\rangle.
\]
Hence, if $\langle f,R_Tf\rangle\ge0$, the extension can be chosen so that
$\langle \widetilde f,R_{T'}\widetilde f\rangle\ge0$. Since
$\|\widetilde f\|^2>0$, this gives a nonnegative Rayleigh quotient for
$R_{T'}$. If the original Rayleigh quotient is strictly positive, the same
argument yields a strictly positive Rayleigh quotient for $R_{T'}$.
\end{proof}

\begin{corollary}[Strict crossing of the zero level set]\label{cor:strict-zero}
Let $T'$ be obtained from $T$ by attaching one pendant edge at an arbitrary
vertex of $T$. If $\lambda_{\max}(R_T)=0$, then
\[
  \lambda_{\max}(R_{T'})>0.
\]
\end{corollary}

\begin{proof}
Take the strictly positive Perron vector $f$ of $R_T$. Then $A>0$ and
$S^2\le dA<(d+2)A$, so
\[
  A-\frac{S^2}{d+2}>0.
\]
Therefore
\[
  \max_y\Delta(y)
  =
  \frac2{d(d+1)}
  \left(A-\frac{S^2}{d+2}\right)
  >0.
\]
Since $\langle f,R_Tf\rangle=0$, the corresponding extension has strictly
positive Rayleigh quotient for $R_{T'}$, and hence
\[
  \lambda_{\max}(R_{T'})>0.
\]
\end{proof}

\begin{example}[Stars]\label{ex:stars}
Let $S_k$ be the star with $k$ edges. Then
\[
  R_{S_k}=-\left(1+\frac2k\right)I+\frac1kJ,
\]
where $J$ is the $k\times k$ all-ones matrix. Hence
\[
  \lambda_{\max}(R_{S_k})=-\frac2k,
\]
with Perron eigenvector ${\bf 1}$.
\end{example}

Each edge of $S_k$ joins the central vertex of degree $k$ to a leaf of degree
$1$. Hence each diagonal entry of $R_{S_k}$ is
\[
  -1-\frac1k,
\]
and two distinct edges meet only at the center, contributing $1/k$ to the
off-diagonal entry. Therefore
$R_{S_k}=-(1+2/k)I+J/k$. The vector ${\bf 1}$ is an eigenvector with eigenvalue
\[
  -1-\frac1k+\frac{k-1}{k}=-\frac2k.
\]

\begin{example}[Double stars {\cite[Example 2]{baihua-dem}}]\label{ex:double-star}
Let $T_2(a,b)$ be the double star whose two spine vertices carry $a$ and $b$
pendant edges. Then
\[
  \lambda_{\max}(R_{T_2(a,b)})<0
  \quad\Longleftrightarrow\quad
  (a-1)(b-1)<4,
\]
and equality holds exactly when $(a-1)(b-1)=4$.
\end{example}

In particular, up to symmetry, the zero cases are
\[
  (a,b)=(3,3),\ (2,5).
\]
The symmetric case $(3,3)$ will later generate the stable zero family
\[
  (3,0,\ldots,0,3)
\]
by subdividing the central edge.

\subsection{Non-caterpillar reduction}

\begin{lemma}\label{lem:noncat-s32}
Every non-caterpillar tree contains $S_3^2$ as a connected subtree.
\end{lemma}

\begin{proof}
Let $T^\circ$ be the graph obtained from $T$ by deleting all leaves. Since $T$
is not a caterpillar, the graph $T^\circ$ is not a path. Therefore $T^\circ$
contains a vertex $v$ of degree at least three. Choose three distinct neighbors
$u_1,u_2,u_3$ of $v$ in $T^\circ$. Because each $u_j$ still belongs to
$T^\circ$, it has a neighbor $w_j\neq v$ in $T$. The edges
\[
  vu_1,\ vu_2,\ vu_3,\ u_1w_1,\ u_2w_2,\ u_3w_3
\]
form a connected subtree isomorphic to $S_3^2$.
\end{proof}

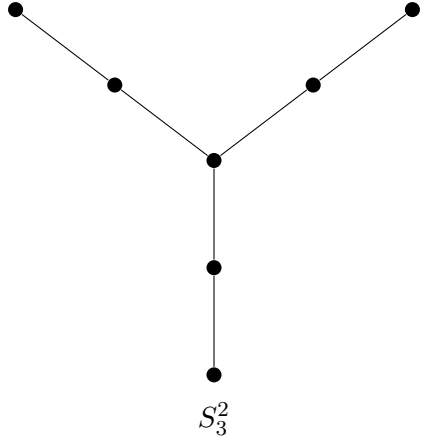
\begin{figure}[t]
\centering
\begin{tikzpicture}[scale=1.05, every node/.style={circle,fill=black,inner sep=2.0pt}]
  \node (c) at (0,0) {};
  \node (a1) at (-1.25,0.95) {};
  \node (a2) at (-2.50,1.90) {};
  \node (b1) at (1.25,0.95) {};
  \node (b2) at (2.50,1.90) {};
  \node (d1) at (0,-1.35) {};
  \node (d2) at (0,-2.70) {};
  \draw (c)--(a1)--(a2);
  \draw (c)--(b1)--(b2);
  \draw (c)--(d1)--(d2);
  \node[draw=none,fill=none,rectangle] at (0,-3.25) {$S_3^2$};
\end{tikzpicture}
\caption{The unique non-caterpillar tree in the zero level set, namely $S_3^2$.}
\end{figure}

\begin{proposition}[The non-caterpillar part of the zero level set]
Let $T$ be a non-caterpillar tree. Then
\[
  \lambda_{\max}(R_T)=0
  \quad\Longleftrightarrow\quad
  T=S_3^2.
\]
In particular, if $T\ne S_3^2$, then $\lambda_{\max}(R_T)>0$.
\end{proposition}

\begin{proof}
The equality $\lambda_{\max}(R_{S_3^2})=0$ is already recorded in
\cite[Definition 5 and the discussion preceding Proposition 4]{baihua-dem}.
Equivalently, that paper gives a positive zero vector on $S_3^2$ with value
$3$ on the three edges incident to the center and value $1$ on the three leaf
edges. Hence $R_{S_3^2}w=0$, and Perron--Frobenius implies that
$\lambda_{\max}(R_{S_3^2})=0$.
By Lemma \ref{lem:noncat-s32}, every non-caterpillar tree $T\ne S_3^2$
properly contains $S_3^2$ as a connected subtree. Since every finite tree can
be obtained from any connected subtree by successively attaching pendant edges,
there exists a finite sequence of trees
\[
  S_3^2=T_0\subset T_1\subset \cdots \subset T_k=T
\]
such that, for each $j=1,\dots,k$, the tree $T_j$ is obtained from $T_{j-1}$
by attaching one pendant edge at some vertex. By Corollary
\ref{cor:strict-zero},
\[
  \lambda_{\max}(R_{T_1})>0.
\]
Applying the strict part of Proposition \ref{prop:upper-closed} inductively
along the sequence, we obtain
\[
  \lambda_{\max}(R_{T_j})>0,\qquad 1\le j\le k.
\]
In particular,
\[
  \lambda_{\max}(R_T)>0.
\]
\end{proof}

\section{Long-Spine Caterpillars}

We now treat the case of long spines. The aim of this section is to show that,
for $m\ge 12$, a caterpillar with $\lambda_{\max}(R_T)<0$ cannot have any
internal pendant edges. Once this has been proved, only the endpoint families
$C_m(a,b)$ remain.

For a caterpillar, the pendant edges attached to a fixed spine vertex form a
single orbit under the automorphism group permuting these sibling leaves.
Because the Perron eigenvector is unique and positive, it is constant on each
such orbit. Hence the spectral top $\lambda_{\max}(R_T)$ is equal to the
spectral top of the symmetric orbit quotient introduced in Section 2.

To exclude internal pendant edges, we begin with the minimal configuration in
which such a defect occurs. For each $2\le i\le m-1$, define
\[
  \eta_i=(1,0,\ldots,0,\underset{i\text{-th position}}{1},0,\ldots,0,1).
\]
Thus the corresponding caterpillar $T_m(\eta_i)$ has exactly three pendant-edge
orbits: one at each endpoint $v_1,v_m$, and one at the internal spine vertex
$v_i$. Any caterpillar with an internal pendant edge at $v_i$ is coordinatewise
larger than this minimal defect.

\begin{figure}[t]
\centering
\begin{tikzpicture}[
  scale=.92,
  vertex/.style={circle,fill=black,inner sep=1.45pt},
  every node/.style={font=\small}
]
  \node[vertex] (v1) at (0,0) {};
  \node[vertex] (v2) at (1.15,0) {};
  \node[vertex] (v3) at (2.30,0) {};
  \node[vertex] (v4) at (4.80,0) {};
  \node[vertex] (v5) at (5.95,0) {};
  \node[vertex] (v6) at (7.10,0) {};
  \draw (v1)--(v2)--(v3);
  \draw[densely dotted,thick] (v3)--(v4);
  \draw (v4)--(v5)--(v6);
  \node at (3.55,.18) {$\cdots$};

  \node[vertex] (l1) at (-.65,.72) {};
  \draw (v1)--(l1);
  \node[vertex] (li) at (3.55,.78) {};
  \draw ($(3.55,0)$)--(li);
  \node[vertex] (lm) at (7.75,.72) {};
  \draw (v6)--(lm);

  \node[draw=none,fill=none,rectangle] at (3.55,-.95)
    {$T_m(\eta_i)$: one pendant-edge orbit at $v_1$, one at $v_i$, and one at $v_m$};
\end{tikzpicture}
\caption{The minimal internal defect $T_m(\eta_i)$.}
\label{fig:minimal-internal-defect}
\end{figure}

The proof strategy for the long-spine case is therefore:
\begin{enumerate}[label=\textup{(\arabic*)}]
\item show that each minimal internal defect $T_m(\eta_i)$ already has
\[
  \lambda_{\max}(R_{T_m(\eta_i)})\ge0
\]
when $m\ge12$;
\item use the upward closure of the nonnegative region to exclude every
caterpillar that contains an internal defect;
\item reduce the classification to the endpoint families $C_m(a,b)$.
\end{enumerate}

\begin{lemma}[Schur determinant for one internal pendant edge]
\label{lem:schur}
Let $m\ge4$ and $2\le i\le m-1$, and let $Q_{m,i}$ be the equivariant quotient matrix for
the caterpillar $T_m(\eta_i)$ in the orthonormal orbit basis of Section 2.
Equivalently, $Q_{m,i}$ is obtained from the symmetric quotient matrix
$Q_m(a)$ by specializing the parameter vector to $\eta_i$. Order the quotient
variables by first listing the $m-1$ spine-edge variables and then the three
pendant-edge orbit variables at $v_1,v_i,v_m$. After eliminating these three
pendant-edge variables, let $S_{m,i}$ denote the Schur complement on the
spine-edge variables. Then
\[
  \det S_{m,i}
  =
  (-1)^{m+1}\frac{B_{m,i}}{81\cdot2^{m-2}},
\]
where
\[
  B_{m,i}=8i^2-8im-8i+13m+20.
\]
If $B_{m,i}\le0$, then $\lambda_{\max}(Q_{m,i})\ge0$. If $B_{m,i}<0$, then
$\lambda_{\max}(Q_{m,i})>0$.
\end{lemma}

\noindent The determinant identity in Lemma~\ref{lem:schur} is also checked by
the symbolic validation script described in Appendix \ref{app:boundary-check}.

\begin{proof}
With the variable ordering specified above, the quotient matrix has the block
form
\[
  Q_{m,i}=
  \begin{pmatrix}
    A & C\\
    C^{\mathsf T} & D
  \end{pmatrix},
\]
where $A$ is the spine-edge block and $D$ is the $3\times3$ block indexed by
the three pendant-edge orbit variables. The pendant block is
\[
  D=\operatorname{diag}\left(-\frac32,-\frac43,-\frac32\right),
\]
which is negative definite. The Schur complement of $D$ in $Q_{m,i}$ is
\[
  S_{m,i}
  =A-CD^{-1}C^{\mathsf T}
  =
  A+\frac16 e_1e_1^{\mathsf T}
    +\frac1{12}(e_{i-1}+e_i)(e_{i-1}+e_i)^{\mathsf T}
    +\frac16 e_{m-1}e_{m-1}^{\mathsf T},
\]
where \(e_1,\ldots,e_{m-1}\) are the standard basis vectors of
\(\mathbb R^{m-1}\). Here the eliminated block \(D\) belongs to the three
orbit variables at \(v_1,v_i,v_m\), not to the full individual pendant-edge
basis. The matrix $A$ is the tridiagonal spine block coming from the general
quotient system. Away from the internal defect, the diagonal entries are $-1$ and the
off-diagonal entries are $1/2$. At the defect one has
\[
  (S_{m,i})_{i-1,i-1}=(S_{m,i})_{i,i}=-\frac34,\qquad
  (S_{m,i})_{i-1,i}=(S_{m,i})_{i,i-1}=\frac5{12},
\]
and the two endpoints have diagonal entry $-5/6$. When $i=2$ or $i=m-1$, the
endpoint correction overlaps with one of the defect rows. In the case $i=2$,
the first two principal minors are obtained directly from the displayed
Schur-complement formula:
\[
  P_1=-\frac{7}{12},\qquad P_2=\frac{19}{72},
\]
which agree with the formulas for $P_{i-1}$ and $P_i$ below after
specializing $i=2$; from $k=3$ onward, the tridiagonal coefficients are the
same as in the generic case until the final endpoint row. The case $i=m-1$ is
the left-right reflection of $i=2$ and gives the same initial values at the
right end.

Let $P_k$ be the leading $k$-th principal minor of $S_{m,i}$. The tridiagonal
recurrence is
\[
  P_k=\alpha_kP_{k-1}-\beta_{k-1}^2P_{k-2},\qquad P_0=1,
\]
where $\alpha_k$ is the $k$-th diagonal entry and $\beta_k$ the $k$-th
off-diagonal entry. Before the defect,
\[
  P_k=(-1)^k\frac{2k+3}{3\cdot 2^k},\qquad 0\le k\le i-2.
\]
Passing through the two defect rows gives
\[
  P_{i-1}=(-1)^{i-1}\frac{2i+3}{3\cdot2^i},
  \qquad
  P_i=(-1)^i\frac{2i+53}{54\cdot2^i}.
\]
To the right of the defect the recurrence again becomes
$P_k=-P_{k-1}-P_{k-2}/4$. Writing
\[
  U_k:=(-1)^k2^kP_k,
\]
this becomes
\[
  U_k=2U_{k-1}-U_{k-2},
\]
whose characteristic polynomial $(r-1)^2$ has a double root at $r=1$. Hence
$U_k$ is affine in $k$, say $U_k=A_0+B_0k$, and therefore
\[
  P_k=\frac{(-1)^k}{2^k}(A_0+B_0k),\qquad i\le k\le m-2,
\]
with
\[
  A_0=\frac{16i^2-24i+53}{54},\qquad
  B_0=\frac{13-8i}{27}.
\]
These constants are determined by the two initial values $P_{i-1}$ and $P_i$.
The final endpoint row has diagonal entry $-5/6$, and substitution gives
\[
  P_{m-1}
  =
  (-1)^{m+1}
  \frac{8i^2-8im-8i+13m+20}{81\cdot 2^{m-2}}.
\]
This is the displayed determinant formula. By the inertia formula,
\[
  \operatorname{In}(Q_{m,i})
  =
  \operatorname{In}(D)+\operatorname{In}(S_{m,i}).
\]
Here \(\operatorname{In}(M)\) denotes the inertia of a real symmetric matrix
\(M\), namely the triple consisting of the numbers of positive, negative, and
zero eigenvalues.
Since $D$ is negative definite, the matrix $Q_{m,i}$ has a nonnegative
eigenvalue as soon as $S_{m,i}$ is not negative definite. Now, if $S_{m,i}$
were negative definite, the sign of its determinant would be
$(-1)^{m-1}$. When $B_{m,i}\le0$, the determinant is zero or has the opposite
sign, so $S_{m,i}$ is not negative definite. Hence $Q_{m,i}$ has a nonnegative
eigenvalue. If $B_{m,i}<0$, then $\det S_{m,i}\ne0$, hence also
$\det Q_{m,i}\ne0$ because $D$ is invertible. Therefore $0$ is not an
eigenvalue of $Q_{m,i}$, and the already obtained nonnegative eigenvalue is in
fact strictly positive.
\end{proof}

\begin{corollary}\label{cor:long-internal}
Let $m\ge12$, and let $T_m(a)$ be a caterpillar. If
\[
  \lambda_{\max}(R_{T_m(a)})<0,
\]
then
\[
  a_i=0,\qquad 2\le i\le m-1.
\]
\end{corollary}

\begin{proof}
For $2\le i\le m-1$, the quadratic polynomial
\[
  B_{m,i}=8i^2-8im-8i+13m+20
\]
is convex in $i$, hence attains its maximum on the interval
\[
  2\le i\le m-1
\]
at one of the two endpoints. A direct computation gives
\[
  B_{m,2}=36-3m,\qquad B_{m,m-1}=36-3m,
\]
and therefore
\[
  B_{m,i}\le0,\qquad 2\le i\le m-1,
\]
for every $m\ge12$. By Lemma \ref{lem:schur}, each minimal internal defect
$T_m(\eta_i)$ then satisfies
\[
  \lambda_{\max}(R_{T_m(\eta_i)})\ge0.
\]
Finally, Proposition \ref{prop:upper-closed} implies that the nonnegative
region is upward closed. Hence every caterpillar containing an internal defect
is also nonnegative. Therefore a caterpillar with
\[
  \lambda_{\max}(R_{T_m(a)})<0
\]
cannot have any internal pendant edge, that is,
\[
  a_i=0,\qquad 2\le i\le m-1.
\]
\end{proof}

\begin{corollary}[Long-spine zero-layer reduction]\label{cor:long-zero-reduction}
Let $m\ge13$, and let $T_m(a)$ be a caterpillar. If $a_i>0$ for some
$2\le i\le m-1$, then
\[
  \lambda_{\max}(R_{T_m(a)})>0.
\]
Consequently, every caterpillar with $m\ge13$ and
$\lambda_{\max}(R_{T_m(a)})\le0$ is endpoint-only.
\end{corollary}

\begin{proof}
If $a_i>0$ for some internal index $i$, then $T_m(a)$ contains the minimal
internal defect $T_m(\eta_i)$. For $m\ge13$, the endpoint computation above
gives $B_{m,i}\le 36-3m<0$, so Lemma \ref{lem:schur} yields
\[
  \lambda_{\max}(R_{T_m(\eta_i)})>0.
\]
Applying the strict part of Proposition \ref{prop:upper-closed} along any
sequence of leaf attachments from $T_m(\eta_i)$ to $T_m(a)$ gives
\[
  \lambda_{\max}(R_{T_m(a)})>0.
\]
The final assertion is immediate.
\end{proof}

\subsection{Endpoint families}

\begin{proposition}[Endpoint-only families]\label{prop:endpoint-only}
Let
\[
  C_m(a,b)=T_m(a,0,\ldots,0,b).
\]
For $m\ge12$,
\[
  \lambda_{\max}(R_{C_m(a,b)})<0
  \quad\Longleftrightarrow\quad
  1\le a,b\le3,\quad (a,b)\ne(3,3).
\]
Moreover, $C_m(3,3)$ satisfies $\lambda_{\max}(R_{C_m(3,3)})=0$ for every
$m\ge2$.
\end{proposition}

\begin{proof}
We prove the assertion directly from the caterpillar quotient system in
Section 2, specialized to the endpoint-only parameter vector
\[
  (a,0,\ldots,0,b).
\]
At energy $\lambda=0$, we eliminate the two endpoint pendant-edge orbit
variables and study the resulting Schur complement on the spine-edge variables.
Let $n=m-1$ be the number of spine edges. In the symmetric orbit quotient, the
two eliminated endpoint orbit variables have diagonal entries $-3/(a+1)$ and
$-3/(b+1)$, so the eliminated block is negative definite. The
sibling-difference directions in the full edge space have eigenvalues
$-(a+3)/(a+1)$ and $-(b+3)/(b+1)$, respectively, and are already negative;
hence they do not affect the spectral top.

The Schur complement on the spine-edge variables is the tridiagonal matrix
$S_n(a,b)$ with off-diagonal entries $1/2$, interior diagonal entries $-1$, and
endpoint diagonal entries
\[
  \alpha_a=-\frac{a+9}{6(a+1)},\qquad
  \alpha_b=-\frac{b+9}{6(b+1)}.
\]
Thus $C_m(a,b)$ has $\lambda_{\max}(R_{C_m(a,b)})<0$ if and only if $S_n(a,b)$ is
negative definite.

Let $P_k$ be the leading $k$-th principal minor of $S_n(a,b)$. For
$1\le k\le n-1$,
\[
  P_k=(-1)^k2^{-k}r_k,\qquad
  r_k=1+\frac{k}{3}\rho_a,
  \qquad
  \rho_a=\frac{2(3-a)}{a+1}.
\]
For the final minor,
\[
  P_n=(-1)^n2^{-n}t_n,\qquad
  t_n=\frac13\left[
    \rho_a-\sigma_b\left(1+\frac{n-1}{3}\rho_a\right)
  \right],
\]
where
\[
  \sigma_b=\frac{2(b-3)}{b+1}.
\]
These formulae follow from the recurrence
$P_k=\alpha_kP_{k-1}-(1/2)^2P_{k-2}$, with the last step using $\alpha_b$
instead of the interior diagonal $-1$.

By Sylvester's criterion, $S_n(a,b)$ is negative definite if and only if
$(-1)^kP_k>0$ for all $1\le k\le n$. Suppose $m\ge12$, so $n-1\ge10$. If
$a\ge4$, then $\rho_a\le -2/5$, and hence
\[
  r_{n-1}=1+\frac{n-1}{3}\rho_a
  \le 1-\frac{10}{3}\cdot\frac25<0,
\]
so the matrix is not negative definite. By symmetry, the same conclusion holds
if $b\ge4$.

It remains to consider $1\le a,b\le3$. Then $\rho_a\ge0$ and
$\sigma_b\le0$, so every $r_k$ is positive and $t_n>0$ unless
$\rho_a=\sigma_b=0$, which is precisely $(a,b)=(3,3)$. Hence all these endpoint
families are negative except $(3,3)$.

For $(a,b)=(3,3)$, the eliminated spine matrix has endpoint diagonals $-1/2$,
interior diagonals $-1$, and off-diagonals $1/2$. The constant spine vector
lies in its kernel. Equivalently, in the original tree, putting the spine-edge
weights equal to $3$ and all pendant-edge weights equal to $1$ gives
$R_Tw=0$. Since $w>0$, Perron--Frobenius gives $\lambda_{\max}(R_T)=0$.
\end{proof}

\begin{corollary}[Long-spine zero endpoint family]\label{cor:long-zero-endpoint}
Let $m\ge13$. Among endpoint-only caterpillars $C_m(a,b)$, one has
\[
  \lambda_{\max}(R_{C_m(a,b)})=0
  \quad\Longleftrightarrow\quad
  (a,b)=(3,3).
\]
Equivalently, among endpoint-only caterpillars with $m\ge13$, the zero level
set consists exactly of the family $(3,0,\ldots,0,3)$.
\end{corollary}

\begin{proof}
Let $n=m-1\ge12$ and let $S_n(a,b)$ be the Schur complement from the proof of
Proposition \ref{prop:endpoint-only}. If
\[
  \lambda_{\max}(R_{C_m(a,b)})=0,
\]
then the symmetric orbit matrix $Q_m(a,b)$ is negative semidefinite. Since the
eliminated endpoint blocks are negative definite, the Schur complement
$S_n(a,b)$ is also negative semidefinite. Since every principal submatrix of a
negative semidefinite matrix is again negative semidefinite, each leading
principal submatrix of $S_n(a,b)$ has determinant $P_k$, and therefore
\[
  (-1)^kP_k\ge0,\qquad 1\le k\le n.
\]

Suppose first that $a\ge4$. Then \(\rho_a\le -2/5\), so
\[
  r_{n-1}=1+\frac{n-1}{3}\rho_a
  \le 1-\frac{11}{3}\cdot\frac25<0
\]
because \(n-1\ge11\). Hence
\[
  (-1)^{n-1}P_{n-1}=2^{-(n-1)}r_{n-1}<0,
\]
contrary to negative semidefiniteness. By symmetry, one also cannot have
$b\ge4$. Thus any endpoint-only zero example with $m\ge13$ must satisfy
$1\le a,b\le3$.

For \(1\le a,b\le3\), Proposition \ref{prop:endpoint-only} already shows that
every family other than \((a,b)=(3,3)\) is strictly negative, while
\((3,3)\) has zero largest eigenvalue. Therefore
\[
  \lambda_{\max}(R_{C_m(a,b)})=0
  \quad\Longleftrightarrow\quad
  (a,b)=(3,3).
\]
\end{proof}

For long spines, Corollary \ref{cor:long-internal} leaves only these endpoint
families. Proposition \ref{prop:endpoint-only} therefore gives the stable
zero-eigenvalue family
\[
  (3,0,\ldots,0,3).
\]

\begin{remark}
Once the zero-eigenvalue family $(3,0,\ldots,0,3)$ is known, Proposition
\ref{prop:upper-closed} and Corollary \ref{cor:strict-zero} already determine
the sign of all endpoint families that are coordinatewise comparable with it.
However, this does not cover the whole endpoint region, because families such as
$C_m(4,1)$ or $C_m(5,3)$ are not coordinatewise comparable with $C_m(3,3)$.
For this reason, Proposition \ref{prop:endpoint-only} cannot be replaced by a
direct corollary of the zero-eigenvalue family and requires the Schur-complement analysis
above.
\end{remark}

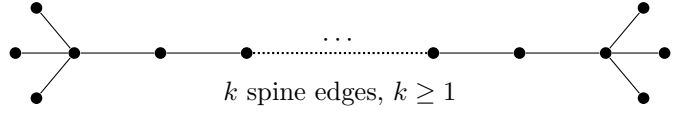
\begin{figure}[t]
\centering
\begin{tikzpicture}[
  scale=.95,
  vertex/.style={circle,fill=black,inner sep=1.55pt},
  every node/.style={font=\small}
]
  \node[vertex] (u) at (0,0) {};
  \node[vertex] (p1) at (1.2,0) {};
  \node[vertex] (p2) at (2.4,0) {};
  \node[vertex] (q1) at (5.0,0) {};
  \node[vertex] (q2) at (6.2,0) {};
  \node[vertex] (v) at (7.4,0) {};
  \draw (u)--(p1)--(p2);
  \draw[densely dotted,thick] (p2)--(q1);
  \draw (q1)--(q2)--(v);

  \foreach \ang/\name in {130/Lone,180/Ltwo,230/Lthree} {
    \node[vertex] (\name) at ($(u)+(\ang:.82)$) {};
    \draw (u)--(\name);
  }
  \foreach \ang/\name in {50/Rone,0/Rtwo,-50/Rthree} {
    \node[vertex] (\name) at ($(v)+(\ang:.82)$) {};
    \draw (v)--(\name);
  }
  \node at (3.7,.18) {$\cdots$};
  \node[below=7pt] at (3.7,0) {$k$ spine edges, $k\ge1$};
\end{tikzpicture}
\caption{The stable zero-eigenvalue family $(3,0,\ldots,0,3)$, obtained by subdividing the
central edge of the double star $S_{3,3}$.}
\label{fig:stable-zero-family}
\end{figure}

Together with Corollary \ref{cor:long-internal}, this proves the long-spine
part of Theorem \ref{thm:main}.

\section{Short-Spine Verification}

After the separate treatments of $m=1$, $m=2$, and $m\ge12$, the only
remaining range is $3\le m\le11$. In this finite regime, downward closure
reduces the problem to the maximal negative elements. They are listed in
Table \ref{tab:maximal-negative}. The $m=12$ row is included as the first
instance of the long-spine endpoint rule and as a consistency check for the
finite enumeration.

More explicitly, for a fixed spine order $m$ let
\[
  \mathcal N_m=\{a:\lambda_{\max}(R_{T_m(a)})<0\}.
\]
Proposition \ref{prop:upper-closed} says that the complement of
$\mathcal N_m$ is upward closed, and therefore $\mathcal N_m$ is downward
closed. Hence the finite table is organized by the maximal elements of
$\mathcal N_m$ with respect to the coordinatewise partial order: once these are known, the whole negative region is exactly
their coordinatewise downward closure. To determine the zero level set, we also
look at the first nonnegative boundary
\[
  \partial_+\mathcal N_m
  =
  \{a\notin\mathcal N_m:\text{ every immediate predecessor of }a
  \text{ lies in }\mathcal N_m\}
\]
in the same finite range. This boundary is listed in
Appendix \ref{app:min-boundary}.

The finite verification is carried out by the deterministic procedure below. Here
$\overline a$ denotes the canonical representative of $a$ up to reversal.
\begin{enumerate}[label=\textup{(\arabic*)}]
  \item Start from the minimal canonical parameter
  \[
    a^{(0)}=(1,0,\ldots,0,1).
  \]
  \item For every parameter $a$ reached by the search, compute the sign of
  $\lambda_{\max}(R_{T_m(a)})$.
  \item If $\lambda_{\max}(R_{T_m(a)})<0$, record $a$ as negative and add all
  canonical children
  \[
    \overline{a+e_i},\qquad 1\le i\le m,
  \]
  to the search queue.
  \item If $\lambda_{\max}(R_{T_m(a)})\ge0$, stop in that direction; by
  Proposition \ref{prop:upper-closed}, every coordinatewise larger parameter is
  also nonnegative.
\end{enumerate}
Since $\mathcal N_m$ is downward closed, every negative parameter is reached by
this procedure: along any coordinatewise path from $(1,0,\ldots,0,1)$ to a
negative parameter, all intermediate parameters are also negative. When the
procedure terminates, the recorded negative parameters with no negative child
are exactly the maximal elements of $\mathcal N_m$.

For a fixed caterpillar parameter $a$, the restricted operator from Section 2
has rational entries in the coordinate basis \((x_j,y_i)\). The symmetric
orthonormal orbit matrix \(Q_m(a)\) from Section 2 is obtained from this
coordinate matrix by a positive diagonal similarity, so both matrices have the
same spectrum. Concretely, if \(M_m(a)\) denotes the common-value coordinate
matrix and \(D\) is diagonal with entries \(1\) on the spine coordinates and
\(\sqrt{a_i}\) on each pendant-orbit coordinate corresponding to a vertex with
\(a_i>0\), then
\[
  Q_m(a)=DM_m(a)D^{-1}.
\]
The remaining eigenspaces come from differences of sibling pendant edges and
have negative eigenvalues. Hence the inequality
\[
  \lambda_{\max}(R_{T_m(a)})<0
\]
is equivalent to negative definiteness of the symmetric orbit matrix \(Q_m(a)\).
For the finite verification we instead use an exact rational computation on the
coordinate matrix of this restricted operator. For each parameter \(a\), we form
the characteristic polynomial over \(\mathbb Q\), determine the multiplicity of
the root \(\lambda=0\), and use Sturm root counts to count the roots in the
positive half-line \((0,\infty)\). ThusThus
\[
  \lambda_{\max}(R_{T_m(a)})<0
  \iff
  0\text{ is not a root of }\chi_a\ \text{and}\ \chi_a\text{ has no root in }(0,\infty),
\]
\[
  \lambda_{\max}(R_{T_m(a)})=0
  \iff
  0\text{ is a root of }\chi_a\text{ and }\chi_a\text{ has no root in }(0,\infty),
\]
and
\[
  \lambda_{\max}(R_{T_m(a)})>0
  \iff
  \chi_a\text{ has a root in }(0,\infty).
\]
These alternatives are therefore decided exactly by the zero multiplicity at
\(\lambda=0\) together with the Sturm root count on \((0,\infty)\). Thus, for
each entry $a$ in
Table \ref{tab:maximal-negative}, we check two finite conditions:
\[
  \lambda_{\max}(R_{T_m(a)})<0,
\]
and, for every $1\le i\le m$,
\[
  \lambda_{\max}(R_{T_m(\overline{a+e_i})})\ge0 .
\]
The first condition says that $a$ lies in the negative region; the second says
that it is maximal in that region. The same exact rational sign computation is
used for the first nonnegative boundary. The entries in
Table \ref{tab:zero-exceptions} have explicit positive null vectors, listed in
Appendix \ref{app:zero-vectors}. Since these null vectors are strictly positive,
Perron--Frobenius implies that the corresponding zero eigenvalue is the largest
eigenvalue. The other boundary points in Appendix \ref{app:min-boundary} have
strictly positive largest eigenvalue.

Thus the finite verification is exact and reproducible. It uses exact rational
arithmetic and Sturm root counts rather than floating-point tests, together
with the monotonicity theorem as the stopping rule. Every child
\[
  \overline{a+e_i}
\]
is checked explicitly by an exact rational computation on the quotient matrix.
The verification code and output files are described in
Appendix \ref{app:boundary-check}.

\begin{table}[t]
\centering
\small
\begin{tabular}{c p{.78\linewidth}}
\toprule
$m$&Maximal negative elements\\
\midrule
3&
$(1,1,2)$, $(1,3,1)$, $(2,0,4)$, $(1,0,8)$\\
4&
$(1,0,2,1)$, $(1,1,0,2)$, $(1,0,0,5)$, $(2,0,0,4)$\\
5&
$(1,0,1,0,1)$, $(1,1,0,0,2)$, $(1,0,0,0,4)$, $(2,0,0,0,3)$\\
6&
$(1,0,0,1,0,1)$, $(1,1,0,0,0,2)$, $(1,0,0,0,0,4)$, $(2,0,0,0,0,3)$\\
7&
$(1,1,0,0,0,0,2)$, $(1,0,0,0,0,0,4)$, $(2,0,0,0,0,0,3)$\\
8&
$(1,1,0,0,0,0,0,2)$, $(2,0,0,0,0,0,0,3)$\\
9&
$(1,0,0,0,0,0,0,1,1)$, $(2,0,0,0,0,0,0,0,3)$\\
10&
$(1,0,0,0,0,0,0,0,1,1)$, $(2,0,0,0,0,0,0,0,0,3)$\\
11&
$(1,0,0,0,0,0,0,0,0,1,1)$, $(2,0,0,0,0,0,0,0,0,0,3)$\\
12&
$(2,0,0,0,0,0,0,0,0,0,0,3)$\\
\bottomrule
\end{tabular}
\caption{Maximal elements of the negative caterpillar downsets. Each listed
parameter satisfies $\lambda_{\max}<0$, whereas increasing any single
coordinate yields $\lambda_{\max}\ge0$. For $3\le m\le11$, the full negative
region is the coordinatewise downward closure of these maximal elements; the
$m=12$ row matches the long-spine classification.}
\label{tab:maximal-negative}
\end{table}

\begin{table}[t]
\centering
\small
\begin{tabular}{c l}
\toprule
$m$&Exceptional caterpillar zero parameters\\
\midrule
2&$(2,5)$\\
3&$(1,4,1)$, $(1,0,9)$\\
4&$(1,0,0,6)$\\
5&$(2,0,0,0,4)$, $(1,0,0,0,5)$\\
8&$(1,0,0,0,0,0,0,4)$\\
9&$(1,1,0,0,0,0,0,0,2)$\\
12&$(1,0,0,0,0,0,0,0,0,0,1,1)$\\
\bottomrule
\end{tabular}
\caption{Finite exceptional caterpillars in the zero level set. Together with
the stable family $(3,0,\ldots,0,3)$, they exhaust all caterpillars with
$\lambda_{\max}(R_T)=0$.}
\label{tab:zero-exceptions}
\end{table}

\subsection{Zero examples}

After the long-spine analysis, only the short spines
\[
  3\le m\le12
\]
must be checked. In this range, Proposition \ref{prop:upper-closed} controls
the search space, so the negative region is determined by the maximal elements
in Table \ref{tab:maximal-negative}, while the zero classification is read off
from the first nonnegative boundary recorded in
Appendix \ref{app:min-boundary}. Indeed, if an immediate predecessor of a zero
parameter were nonnegative, then adding the corresponding leaf would make the
new parameter strictly positive by Proposition \ref{prop:upper-closed} and
Corollary \ref{cor:strict-zero}, contradicting \(\lambda_{\max}=0\).

The double-star family gives an illustrative model for the phase transition. By
Example \ref{ex:double-star}, the sign is governed by the single
threshold
\[
  \lambda_{\max}(R_{T_2(a,b)})<0 \iff (a-1)(b-1)<4,
\]
with equality
\[
  \lambda_{\max}(R_{T_2(a,b)})=0 \iff (a-1)(b-1)=4.
\]
Thus the zero-eigenvalue points are
\[
  (a,b)=(3,3),\ (2,5),\ (5,2),
\]
up to the symmetry interchanging the two ends. This is the first appearance of
the stable family: subdividing the central edge of $(3,3)$ gives the zero
caterpillars
\[
  (3,0,\ldots,0,3).
\]

\begin{figure}[htbp]
\centering
\begin{tikzpicture}[scale=.74]
  \draw[->] (0.6,0.6)--(7.8,0.6) node[right] {$a$};
  \draw[->] (0.6,0.6)--(0.6,7.8) node[above] {$b$};
  \foreach \x in {1,...,7} {
    \draw (\x,0.54)--(\x,0.66) node[below=3pt] {\scriptsize \x};
    \draw (0.54,\x)--(0.66,\x) node[left=3pt] {\scriptsize \x};
  }
  \draw[thick] plot[domain=1.67:7.4,samples=90]
    (\x,{1+4/(\x-1)});
  \node[align=center] at (2.15,1.95) {\scriptsize $\lambda_{\max}<0$};
  \node[align=center] at (5.85,5.95) {\scriptsize $\lambda_{\max}>0$};
  \node[align=center] at (4.85,2.75) {\scriptsize $(a-1)(b-1)=4$};
  \foreach \p/\q in {3/3,2/5,5/2} {
    \fill (\p,\q) circle (2.2pt);
  }
  \node[font=\scriptsize,above right=2pt] at (3,3) {$(3,3)$};
  \node[font=\scriptsize,above left=2pt] at (2,5) {$(2,5)$};
  \node[font=\scriptsize,below right=2pt] at (5,2) {$(5,2)$};
\end{tikzpicture}
\caption{Phase diagram for the double-star family $T_2(a,b)$. The curve
$(a-1)(b-1)=4$ separates the negative and positive regions, and the marked
lattice points are exactly the zero cases.}
\label{fig:double-star-phase}
\end{figure}
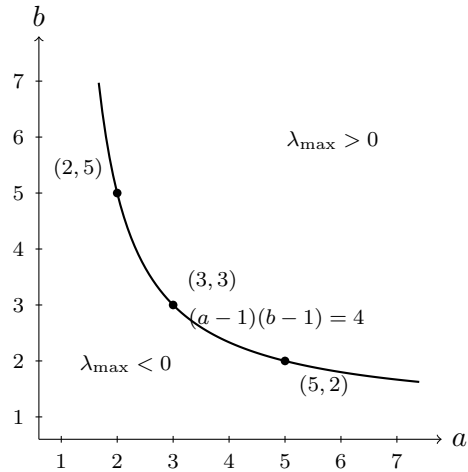

These examples also show that symmetry alone is not the correct organizing
principle for \(\lambda_{\max}\). The symmetric double star \((3,3)\) lies in
the zero level set, whereas highly asymmetric double stars such as
\[
  (1,n-3)
\]
still remain in the negative region.

For the zero level set, the long-spine part is already settled by
Corollaries \ref{cor:long-zero-reduction} and \ref{cor:long-zero-endpoint}, so
every remaining zero caterpillar must lie in the finite range \(m\le12\).
Corollary \ref{cor:strict-zero} places such cases on the first nonnegative
boundary, and the exact verification shows that the zero points there are
precisely those in Table \ref{tab:zero-exceptions}. Therefore the zero level
set consists of the stable family \((3,0,\ldots,0,3)\), the nine exceptional
short-spine caterpillars, and the single non-caterpillar tree \(S_3^2\).

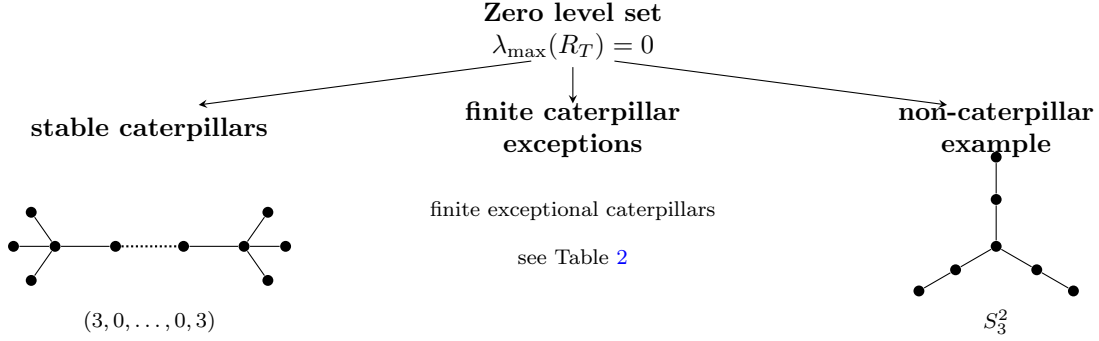
\begin{figure}[htbp]
\centering
\begin{tikzpicture}[
  vertex/.style={circle,fill=black,inner sep=1.45pt},
  title/.style={font=\small\bfseries,align=center},
  note/.style={font=\scriptsize,align=center},
  >=stealth
]
  \node[title] at (0,2.8) {Zero level set\\$\lambda_{\max}(R_T)=0$};

  \node[title] at (-5.6,1.5) {stable caterpillars};
  \node[title] at (0,1.5) {finite caterpillar\\exceptions};
  \node[title] at (5.6,1.5) {non-caterpillar\\example};

  \draw[->] (-.55,2.4)--(-4.95,1.82);
  \draw[->] (0,2.32)--(0,1.84);
  \draw[->] (.55,2.4)--(4.95,1.82);

  \begin{scope}[shift={(-5.6,-0.05)}]
    \node[vertex] (u) at (-1.25,0) {};
    \node[vertex] (x) at (-.45,0) {};
    \node[vertex] (y) at (.45,0) {};
    \node[vertex] (v) at (1.25,0) {};
    \draw (u)--(x);
    \draw[densely dotted,thick] (x)--(y);
    \draw (y)--(v);
    \foreach \ang/\name in {125/a,180/b,235/c} {
      \node[vertex] (\name) at ($(u)+(\ang:.55)$) {};
      \draw (u)--(\name);
    }
    \foreach \ang/\name in {55/d,0/e,-55/f} {
      \node[vertex] (\name) at ($(v)+(\ang:.55)$) {};
      \draw (v)--(\name);
    }
    \node[note] at (0,-1.0) {$(3,0,\ldots,0,3)$};
  \end{scope}

  \begin{scope}[shift={(0,0)}]
    \node[note] at (0,.42) {finite exceptional caterpillars};
    \node[note] at (0,-.18) {see Table \ref{tab:zero-exceptions}};
  \end{scope}

  \begin{scope}[shift={(5.6,-0.05)}]
    \node[vertex] (c) at (0,0) {};
    \foreach \ang/\name in {90/a,210/b,330/d} {
      \node[vertex] (\name1) at (\ang:.62) {};
      \node[vertex] (\name2) at (\ang:1.18) {};
      \draw (c)--(\name1)--(\name2);
    }
    \node[note] at (0,-1.0) {$S_3^2$};
  \end{scope}
\end{tikzpicture}
\caption{Schematic structure of the zero level set. Its caterpillar part
consists of the stable family $(3,0,\ldots,0,3)$ together with the nine
finite exceptional parameters in Table \ref{tab:zero-exceptions}, while its
non-caterpillar part consists only of $S_3^2$.}
\label{fig:zero-layer}
\end{figure}

\FloatBarrier

\section{Infinite Trees}

The classification above is intrinsically finite. If $T$ is an infinite locally
finite tree, the matrix $R_T$ should be interpreted as an operator on edge
functions. When the degrees are uniformly bounded, it defines a bounded
self-adjoint operator on $\ell^2(E(T))$. More generally, one can still consider
the Rayleigh quotient on finitely supported edge functions and set
\[
  \Lambda_c(T)
  =
  \sup_{0\ne f\in C_c(E(T))}
  \frac{\langle f,R_Tf\rangle}{\langle f,f\rangle}.
\]
Here \(C_c(E(T))\) denotes the space of finitely supported real-valued
functions on the edge set \(E(T)\).
For bounded-degree trees, $\Lambda_c(T)=\sup\sigma(R_T)$.

\begin{proposition}[Infinite trees]
Let $T$ be an infinite locally finite tree. Then
\[
  \Lambda_c(T)\ge0.
\]
In particular, if $R_T$ is a bounded self-adjoint operator on $\ell^2(E(T))$,
then $\sup\sigma(R_T)\ge0$.
\end{proposition}

\begin{proof}
Every infinite locally finite tree contains a one-sided infinite simple
path
\[
  e_1,e_2,e_3,\ldots .
\]
For $N\ge1$, define $f_N(e_j)=1$ for $1\le j\le N$ and $f_N(e)=0$ otherwise.
By the vertex decomposition, every interior vertex of the selected path segment
has exactly two incident selected edges, and therefore its local contribution is
\[
  \frac1{d_v}(2^2-2\cdot2)=0.
\]
Only the two endpoints of the segment contribute, and each endpoint contribution
is bounded below by $-1$. Hence
\[
  \langle f_N,R_Tf_N\rangle\ge -2,\qquad
  \|f_N\|^2=N.
\]
Therefore the Rayleigh quotients of $f_N$ tend to $0$ from below, and
$\Lambda_c(T)\ge0$.
\end{proof}

\begin{remark}[Infinite trees]
The proposition shows that no infinite locally finite tree has a negative
spectral top in this sense.
For example, the infinite path lies in the zero level set: its spectral top is $0$.
Finite positive-curvature trees can converge to the zero level set when the
spine order tends to infinity; the stable finite family
$(3,0,\ldots,0,3)$ is one such manifestation. The usual finite statement that a
tree in the zero level set becomes positive after attaching a leaf uses a positive Perron
eigenvector and does not directly apply to the edge of the essential spectrum
on an infinite path. The half-line endpoint models
$(a,0,0,\ldots)$ therefore require a separate one-dimensional analysis.
\end{remark}

\section*{Acknowledgements}

The author is grateful to Bobo Hua for his support, and to Shuliang Bai for
helpful discussions and suggestions.

\appendix
\section{Minimal Nonnegative Boundary}\label{app:min-boundary}

The following table records the minimal nonnegative boundary
\[
  \partial_+\mathcal N_m
\]
for
\[
  3\le m\le12.
\]
Together with Table \ref{tab:maximal-negative}, it provides the exact finite
data needed for the short-spine classification. Appendix
\ref{app:zero-vectors} then isolates the zero-eigenvalue entries from this
boundary and records representative positive null vectors. A machine-readable
version of this boundary table is included in the repository described in
Appendix \ref{app:boundary-check}.

\begin{center}
\scriptsize
\begin{longtable}{c p{.82\linewidth}}
\toprule
$m$&Minimal nonnegative boundary\\
\midrule
3&
$(1,1,3)$, $(1,2,2)$, $(2,1,2)$, $(1,4,1)$, $(3,0,3)$,
$(2,0,5)$, $(1,0,9)$\\
4&
$(1,0,1,2)$, $(1,1,1,1)$, $(1,0,3,1)$, $(1,1,0,3)$,
$(1,2,0,2)$, $(3,0,0,3)$, $(1,0,0,6)$, $(2,0,0,5)$\\
5&
$(1,0,0,1,2)$, $(1,0,0,2,1)$, $(1,0,1,0,2)$, $(1,0,1,1,1)$,
$(1,0,2,0,1)$, $(1,1,0,1,1)$, $(1,1,0,0,3)$, $(1,0,0,0,5)$,
$(2,0,0,0,4)$, $(3,0,0,0,3)$\\
6&
$(1,0,0,0,1,2)$, $(1,0,0,0,2,1)$, $(1,0,0,1,0,2)$, $(1,0,0,1,1,1)$,
$(1,0,0,2,0,1)$, $(1,0,1,0,0,2)$, $(1,0,1,0,1,1)$, $(1,0,1,1,0,1)$,
$(1,1,0,0,1,1)$, $(1,1,0,0,0,3)$, $(1,0,0,0,0,5)$, $(2,0,0,0,0,4)$,
$(3,0,0,0,0,3)$\\
7&
$(1,0,0,0,1,0,1)$, $(1,0,0,1,0,0,1)$, $(1,0,0,0,0,1,2)$, $(1,0,0,0,0,2,1)$,
$(1,1,0,0,0,1,1)$, $(1,1,0,0,0,0,3)$, $(1,0,0,0,0,0,5)$, $(2,0,0,0,0,0,4)$,
$(3,0,0,0,0,0,3)$\\
8&
$(1,0,0,0,0,1,0,1)$, $(1,0,0,0,1,0,0,1)$, $(1,0,0,0,0,0,1,2)$,
$(1,0,0,0,0,0,2,1)$, $(1,1,0,0,0,0,1,1)$, $(1,0,0,0,0,0,0,4)$,
$(1,1,0,0,0,0,0,3)$, $(3,0,0,0,0,0,0,3)$\\
9&
\begin{tabular}[t]{@{}l@{}}
$(1,0,0,0,0,0,1,0,1)$, $(1,0,0,0,0,1,0,0,1)$, $(1,0,0,0,1,0,0,0,1)$,\\
$(1,0,0,0,0,0,0,1,2)$, $(1,0,0,0,0,0,0,2,1)$, $(1,1,0,0,0,0,0,0,2)$,\\
$(1,1,0,0,0,0,0,1,1)$, $(1,0,0,0,0,0,0,0,4)$, $(3,0,0,0,0,0,0,0,3)$
\end{tabular}\\
10&
\begin{tabular}[t]{@{}l@{}}
$(1,0,0,0,0,0,0,1,0,1)$, $(1,0,0,0,0,0,1,0,0,1)$, $(1,0,0,0,0,1,0,0,0,1)$,\\
$(1,0,0,0,1,0,0,0,0,1)$, $(1,0,0,0,0,0,0,0,1,2)$, $(1,0,0,0,0,0,0,0,2,1)$,\\
$(1,1,0,0,0,0,0,0,0,2)$, $(1,1,0,0,0,0,0,0,1,1)$, $(1,0,0,0,0,0,0,0,0,4)$,\\
$(3,0,0,0,0,0,0,0,0,3)$
\end{tabular}\\
11&
\begin{tabular}[t]{@{}l@{}}
$(1,0,0,0,0,0,0,0,1,0,1)$, $(1,0,0,0,0,0,0,1,0,0,1)$, $(1,0,0,0,0,0,1,0,0,0,1)$,\\
$(1,0,0,0,0,1,0,0,0,0,1)$, $(1,0,0,0,1,0,0,0,0,0,1)$, $(1,0,0,0,0,0,0,0,0,1,2)$,\\
$(1,0,0,0,0,0,0,0,0,2,1)$, $(1,1,0,0,0,0,0,0,0,0,2)$, $(1,1,0,0,0,0,0,0,0,1,1)$,\\
$(1,0,0,0,0,0,0,0,0,0,4)$, $(3,0,0,0,0,0,0,0,0,0,3)$
\end{tabular}\\
12&
\begin{tabular}[t]{@{}l@{}}
$(1,0,0,0,0,0,0,0,0,0,1,1)$, $(1,0,0,0,0,0,0,0,0,1,0,1)$,\\
$(1,0,0,0,0,0,0,0,1,0,0,1)$, $(1,0,0,0,0,0,0,1,0,0,0,1)$,\\
$(1,0,0,0,0,0,1,0,0,0,0,1)$, $(1,0,0,0,0,0,0,0,0,0,0,4)$,\\
$(3,0,0,0,0,0,0,0,0,0,0,3)$
\end{tabular}\\
\bottomrule
\end{longtable}
\end{center}

\section{Zero Diagrams and Zero Vectors}

The nine finite caterpillar zero exceptions from
Table \ref{tab:zero-exceptions} are displayed below. In each diagram, the
horizontal path is the spine, and the pendant edges are attached at the
indicated spine vertices.

\begin{figure}[htbp]
\centering
\begin{tikzpicture}[vertex/.style={circle,fill=black,inner sep=1.25pt}]
  \begin{scope}[shift={(-4.9,4.2)}]
    \drawsmallcat{$T_2(2,5)$}{2}{2,5}
  \end{scope}
  \begin{scope}[shift={(-.8,4.2)}]
    \drawsmallcat{$T_3(1,4,1)$}{3}{1,4,1}
  \end{scope}
  \begin{scope}[shift={(3.2,4.2)}]
    \drawsmallcat{$T_3(1,0,9)$}{3}{1,0,9}
  \end{scope}

  \begin{scope}[shift={(-5.2,2.0)}]
    \drawsmallcat{$T_4(1,0,0,6)$}{4}{1,0,0,6}
  \end{scope}
  \begin{scope}[shift={(-.8,2.0)}]
    \drawsmallcat{$T_5(2,0,0,0,4)$}{5}{2,0,0,0,4}
  \end{scope}
  \begin{scope}[shift={(3.8,2.0)}]
    \drawsmallcat{$T_5(1,0,0,0,5)$}{5}{1,0,0,0,5}
  \end{scope}

  \begin{scope}[shift={(-5.6,-.35)}]
    \drawsmallcat{$T_8(1,0,\ldots,0,4)$}{8}{1,0,0,0,0,0,0,4}
  \end{scope}
  \begin{scope}[shift={(-.8,-.35)}]
    \drawsmallcat{$T_9(1,1,0,\ldots,0,2)$}{9}{1,1,0,0,0,0,0,0,2}
  \end{scope}
  \begin{scope}[shift={(4.4,-.35)}]
    \drawsmallcat{$T_{12}(1,0,\ldots,0,1,1)$}{12}{1,0,0,0,0,0,0,0,0,0,1,1}
  \end{scope}
\end{tikzpicture}
\caption{The nine finite caterpillar zero exceptions from
Table \ref{tab:zero-exceptions}.}
\label{fig:finite-zero-caterpillars}
\end{figure}
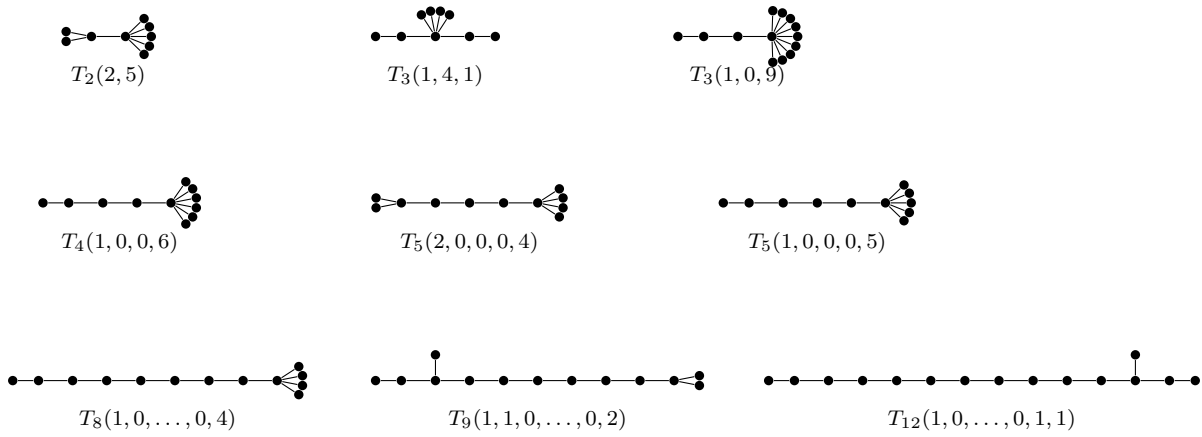

\subsection{Representative Zero Vectors}\label{app:zero-vectors}

We use the following edge order: first list the spine edges
\[
  s_j=v_jv_{j+1},\qquad 1\le j\le m-1,
\]
and then list the pendant edges grouped by the spine vertices
$v_1,\ldots,v_m$.The following positive integer vectors satisfy
\[
  R_{T_m(a)}w=0.
\]
The first eight are listed in the table; the last one is displayed separately
below for readability.
By Perron--Frobenius, this zero eigenvalue is the largest eigenvalue.
For example, in the row $(1,4,1)$, the two spine edges both have weight $6$,
the single leaf at each endpoint has weight $2$, and the four leaves at the
middle spine vertex all have weight $3$.

\begin{center}
\footnotesize
\begin{tabular}{p{.27\linewidth} p{.23\linewidth} p{.38\linewidth}}
\toprule
$a$ & spine-edge weights & pendant-edge weights, grouped by spine vertices\\
\midrule
$(2,5)$ & $(3)$ & $(1,1)\mid(1,1,1,1,1)$\\
$(1,4,1)$ & $(6,6)$ & $(2)\mid(3,3,3,3)\mid(2)$\\
$(1,0,9)$ & $(9,15)$ & $(3)\mid(5,5,5,5,5,5,5,5,5)$\\
$(1,0,0,6)$ & $(9,15,21)$ & $(3)\mid(7,7,7,7,7,7)$\\
$(1,0,0,0,5)$ & $(3,5,7,9)$ & $(1)\mid(3,3,3,3,3)$\\
$(2,0,0,0,4)$ & $(9,11,13,15)$ & $(3,3)\mid(5,5,5,5)$\\
$(1,0,0,0,0,0,0,4)$ & $(3,5,7,9,11,13,15)$ & $(1)\mid(5,5,5,5)$\\
$(1,1,0,0,0,0,0,0,2)$ & $(15,21,19,17,15,13,11,9)$ & $(5)\mid(9)\mid(3,3)$\\
\bottomrule
\end{tabular}
\end{center}

For the last exceptional parameter
\[
  a=(1,0,0,0,0,0,0,0,0,0,1,1),
\]
a positive zero vector is given by spine-edge weights
\[
  (3,5,7,9,11,13,15,17,19,21,15)
\]
and pendant-edge weights
\[
  (1)\mid(9)\mid(5).
\]
The full list of exact zero-vector certificates, in machine-readable form, is
also included in the repository described in Appendix
\ref{app:boundary-check}.

\section{Reproducibility of the Short-Spine Boundary Computation}
\label{app:boundary-check}

This appendix records the exhaustive computer check used for the finite
short-spine classification. It is not an additional classification statement;
it is a reproducibility record for the tables in Section~5.

The verification code and output files are available in the public repository
\begin{center}
\url{https://github.com/coker412/rt-tree-experiments}
\end{center}
The tables in this manuscript were generated from the repository state
identified by the commit
\begin{center}
\texttt{b929c84fd32ac36df789ac7467c42357c0944cfe}.
\end{center}

The exact verification is implemented in
\path{code/exact_short_spine_verification.py}. For each \(3\le m\le12\), it
enumerates canonical caterpillar parameters, computes the sign of
\(\lambda_{\max}(R_{T_m(a)})\) using exact rational characteristic polynomials
and Sturm root counts, and writes the negative downset, the maximal negative
elements, and the first nonnegative boundary to \path{data/exact_short_spine_verification/}.
The tables in Section~5 and Appendix~\ref{app:min-boundary} are read directly
from these output files.

The script \path{code/validate_paper_examples.py} checks the explicit examples
and zero-vector certificates appearing in the paper. Its output is recorded in \path{data/paper_validation/paper_validation.csv}.

The coverage is exhaustive for two reasons. First, every canonical caterpillar
parameter with fixed \(3\le m\le12\) is obtained from
\[
  (1,0,\ldots,0,1)
\]
by successively adding \(1\) to coordinates and then re-canonicalizing up to
reversal. Second, the negative region is downward closed by
Proposition~\ref{prop:upper-closed}. Hence the enumeration of all negative
children produces the full negative downset and its first nonnegative boundary.

After the full negative downsets have been generated, the zero candidates are
obtained by taking all one-step children of the maximal negative elements in
Table~\ref{tab:maximal-negative}. Every short-spine zero caterpillar lies on
this first nonnegative boundary: otherwise it would have a nonnegative
immediate predecessor, and Corollary~\ref{cor:strict-zero} would force it to
have positive spectral top. The exact Sturm verification then identifies
precisely the zero cases listed in Table~\ref{tab:zero-exceptions}.

\end{document}